\documentclass[11pt]{amsart}

\usepackage{amssymb}
\usepackage{graphicx}
\usepackage[margin=1.5in]{geometry}

\newtheorem{thm}{Theorem}[section]
\newtheorem{prop}[thm]{Proposition}
\newtheorem{lem}[thm]{Lemma}
\newtheorem{cor}[thm]{Corollary}
\newtheorem{con}[thm]{Conjecture}

\theoremstyle{definition}
\newtheorem{defn}[thm]{Definition}

\theoremstyle{remark}

\newtheorem{remark}[thm]{Remark}

\numberwithin{equation}{section}

\newcommand{\R}{\mathbb{R}}  

\newcommand{\N}{\mathbb{N}}

\begin{document}


\title{Percival's Conjecture for the Bunimovich Mushroom Billiard}


\author{Sean P Gomes}
\email{sean.p.gomes@gmail.com}





\begin{abstract}
The Laplace-Beltrami eigenfunctions on a compact Riemannian manifold $M$ whose geodesic billiard flow has mixed character have been conjectured by Percival to split into two complementary families, with all semiclassical mass supported in the completely integrable and ergodic regions of phase space respectively.
In this paper, we consider the Dirichlet Laplacian on a family of mushroom billiards $M_t$ parametrised by the length $t\in(0,2]$ of their rectangular part. We prove that $M_t$ satisfies Percival's conjecture for almost all $t\in(0,2]$, hence providing the first example of a billiard known to satisfy Percival's conjecture.
\end{abstract}


\maketitle


\tableofcontents

\newpage



\section{Introduction}


The Bohr correspondence principle informally asserts that the evolution of a quantum mechanical system coincides with the evolution predicted by classical mechanics in the large scale limit.

One of the settings in which aspects of this correspondence can be made rigorous is that of dynamical billiards, which we shall now outline.

If $(M,g)$ is a compact boundaryless Riemannian manifold, we define \emph{dynamical billiards} on $M$ to be the Hamiltonian flow $\phi_t$ on the cotangent  bundle $T^*M$ of the manifold given by Hamilton's equations
\begin{equation}
\label{hameq}
\dot{x}_j=\frac{\partial H}{\partial \xi_j},\quad \dot{\xi}_j=\frac{\partial H}{\partial x_j}
\end{equation}
for the Hamiltonian $H(x,\xi):=|(x,\xi)|_{g^{-1}}^2$ where $g^{-1}$ is the dual metric tensor.

Since the Hamiltonian is a constant of motion for the flow $\phi_t$, it is natural to restrict the domain of this flow to the cosphere bundle \begin{equation} {S^*M:=\{z=(x,\xi)\in T^*M:|z|_{g^{-1}}=1\}}.\end{equation}

More generally, one can define billiards on compact Riemannian manifolds with piecewise smooth boundary in the sense of Chapter 6 of \cite{cfs}, see also \cite{zelditch-zworski}.

To be precise, we assume that we can smoothly embed $M$ in a boundaryless manifold $\tilde{M}$ of the same dimension and that there exist finitely many smooth functions $f_j\in \mathcal{C}^\infty(\tilde{M})$ such that the following conditions are satisfied.
\begin{enumerate}
\item $df_i|_{f_i^{-1}(0)}\neq 0$,\\
\item $df_i, df_j$ are linearly independent on $f_i^{-1}(0)\cap f_j^{-1}(0)$,\\
\item $M=\{x\in \tilde{M}:f_j(x)\geq 0 \textrm{ for all }j\}.$
\end{enumerate}

We can then write $$\partial M=\cup_j \partial M_j:=\cup_j (f_j^{-1}(0)\cap M)$$ and denote by $\mathcal{S}\subset \partial M$ the set of points that lie in $\partial M_j$ for multiple $j$.\\

We define the broken Hamiltonian flow $\phi_t$ on $S^*M$ locally by extending the boundaryless Hamiltonian flow by reflection at non-tangential and non-singular boundary collisions.

That is, if $\phi_{t_0}(z)=(x,\xi_+)$ with $x\in\partial M\setminus \mathcal{S}$ and $\langle\xi_+, N_x\rangle >0$ where $N_x$ is the outgoing unit normal covector, we extend $\phi_t$ to sufficiently small $t>t_0$ by defining $\phi_t(z)=\phi_{t-t_0}(x,\xi_-)$, where $\xi_-\in S^*_x M$ is the unique covector such that $\xi_+ + \xi_- \in T^*\partial M$ and $\pi(\xi_+)=\pi(\xi_-)$ where $\pi: T^*_{\partial M}M\rightarrow T^*\partial M$ is the canonical projection. We terminate all trajectories that meet $\partial M$ in any other manner.\\
\newpage

There are four subsets $\{\mathcal{B}_j\}_{j=1}^4$ of phase space for this class of manifolds which present an obstruction to obtaining a globally defined broken Hamiltonian flow or to the application of tools from microlocal analysis. We enumerate these sets below.
\begin{enumerate}
\item $\mathcal{B}_1=\{z\in S^*M: \phi_t(z)\in\mathcal{S}\}$
\item $\mathcal{B}_2=\{z\in S^*M: \phi_t(z) \in \partial M \textrm{ for infinitely many }t\textrm{ in a bounded interval}\}$
\item $\mathcal{B}_3=\{z\in S^*M: \phi_t(z)\notin \partial M \textrm{ for any }t>0 \textrm{ or }\phi_t(z)\notin \partial M \textrm{ for any }t<0\}$
\item $\mathcal{B}_4=\{z\in S^*M: \phi_t(z)\textrm{ meets }\partial M \textrm{ tangentially for some }t\in \mathbb{R}.\}$
\end{enumerate}
Removing these sets from our flow domain, we then obtain a globally defined billiard flow on $\mathcal{D}=S^*M\setminus (\cup_{j=1}^4\mathcal{B}_j)$. For manifolds without boundary, we simply take $\mathcal{D}=S^*M$.\\

The canonical symplectic form $d\xi\wedge dx$ on $T^*M$ determines a family of measures $\mu_c$ on each of the energy hypersurfaces
\begin{equation}
\Sigma_c=\{z=(x,\xi)\in T^*M:|z|_{g^{-1}}=c\}
\end{equation}
defined implicitly by
\begin{equation}
\int_a^b\int_{\Sigma_c} f\, d\mu_c \, dc = \int_{|(x,\xi)|_{g^{-1}}\in [a,b]}f\, |d\xi\wedge dx|
\end{equation}
for $f\in \mathcal{C}_c^\infty(T^* M)$.

Upon normalisation of $\mu_1$ we then obtain the \emph{Liouville measure} $\mu_L$ on $S^*M$, which allows us to study the ergodic properties of the billiard flow $\phi_t$.
\begin{remark}
It is shown in Section 6.2 of \cite{cfs} that the sets $\mathcal{B}_1,\mathcal{B}_2$ are of Liouville measure zero, and it is shown in \cite{zelditch-zworski} that the set $\mathcal{B}_4$ is of Liouville measure zero for the class of manifolds considered. That the remaining set $\mathcal{B}_3$ is Liouville null is usually taken as an assumption. In particular, it is clear that this assumption is satisfied by bounded domains in $\mathbb{R}^n$.
\end{remark}
The billiard $M$ is said to be \emph{ergodic} if for $\mu_L$-almost all $z\in \mathcal{D}$ and every ${\mu_L\textrm{-measurable}}$ $A\subseteq S^*M$ we have
\begin{equation}
\lim_{T\rightarrow\infty}\frac{|\{t\in [0,T]:\phi_t(z)\in A\}|}{T}=\mu_L(A).
\end{equation}
The most famous example of an ergodic billiard is the \emph{stadium} in $\mathbb{R}^2$, the ergodicity of which was first studied by Bunimovich \cite{bunimovichstadium}.

Another prototypical example of ergodic billiards is provided by surfaces of constant negative curvature, where ergodicity is a consequence of the hyperbolicity of the flow. A proof of ergodicity in this setting can be found in Hopf \cite{hopf}.

On the other hand, if the billiard flow on $M$ is completely integrable, then individual trajectories are constrained to $n$-dimensional Lagrangian submanifolds specified by the values of the $n$ constants of motion, and certainly do not equidistribute.\\


The quantum mechanical analogue of the system \eqref{hameq} is the evolution of a wave function $\psi\in L^2(M)$ according to the rescaled Schrodinger's equation
\begin{equation}
\label{schrod}
-\Delta_g \psi=i\frac{\partial \psi}{\partial t}
\end{equation}

with boundary conditions to ensure self-adjointness of the Laplacian. We shall work with the most studied and technically easiest choice of Dirichlet boundary conditions.

Since the boundary of $M$ is Lipschitz, it follows that the Laplacian $-\Delta_g$ is self adjoint on $L^2$ when given the standard domain $H^2(M)\cap H_0^1(M)$. Standard spectral theory then shows that $-\Delta_g$ has purely discrete spectrum (counting multiplicity) $\{0<E_1 \leq E_2 \leq \ldots \}\subset \mathbb{R}^+$.

By choosing a corresponding orthonormal basis of eigenfunctions $(u_j)_{j\in\N}$, we can thus separate variables and formally expand solutions to \eqref{schrod} as 
\begin{equation}u(x,t)=\sum_{j=1}^\infty a_j u_j(x)\exp(-iE_jt).\end{equation}
From this equation, we see that the localisation properties of high energy solutions to \eqref{schrod} are encoded in the high energy eigenfunctions of the operator $-\Delta_g$.\\
 
The phase space localisation of the high energy eigenfunctions of $-\Delta_g$ can be described using the calculus of semiclassical pseudodifferential operators, as defined in Chapters 4 and 14 of \cite{zworski}.

To each subsequence of $(u_j)$, we can associate at least one non-negative Radon measure $\mu$ on $S^*M$ which provides a notion of phase space concentration in the semiclassical limit. 

We say that the eigenfunction subsequence $(u_{j_k})$ has unique \emph{semiclassical measure} $\mu$ if 
\begin{equation}
\lim_{k\rightarrow \infty}\langle a(x,E_{j_k}^{-1/2}D)u_{j_k},u_{j_k} \rangle=\int_{S^*M} a(x,\xi)\, d\mu
\end{equation}
for each semiclassical pseudodifferential operator with principal symbol $a$ compactly supported supported away from the boundary of $S^*M$. In Chapter 5 of \cite{zworski}, the existence and basic properties of semiclassical measures are established using the calculus of semiclassical pseudodifferential operators (see also \cite{gerard-leichtnam}).\\

A billiard $M$ is then defined to be \emph{quantum ergodic} if there is a full density subsequence of eigenfunctions $(u_{n_k})$ such that the the Liouville measure on $S^*M$ is the unique semiclassical measure associated to the sequence $u_{n_k}$. This statement can be interpreted as saying that the sequence of eigenfunctions equidistributes in phase space with the possible exception of a sparse subsequence.

It is a celebrated result due to G\'{e}rard--Leichtnam \cite{gerard-leichtnam} and Zelditch--Zworski \cite{zelditch-zworski} that compact Riemannian manifolds  that have ergodic geodesic flow are quantum ergodic. This generalises the earlier results of Schnirelman \cite{schnirelman}, Zelditch \cite{zelditch} and Colin de Verdi\`{e}re \cite{verdiere} in the boundaryless setting.\\

In this paper we consider the family of mushroom billiards $M_{t}=R_t\cup S\subset \R^2$ where $R_t=[-r_1,r_1]\times[-t,0]$ and $S$ is the closed upper semidisk of radius $r_2>r_1$ centred at the origin. We denote the area of $M_t$ by $A(t)$.

\vspace{15pt}
\begin{center}
\includegraphics[scale=0.5]{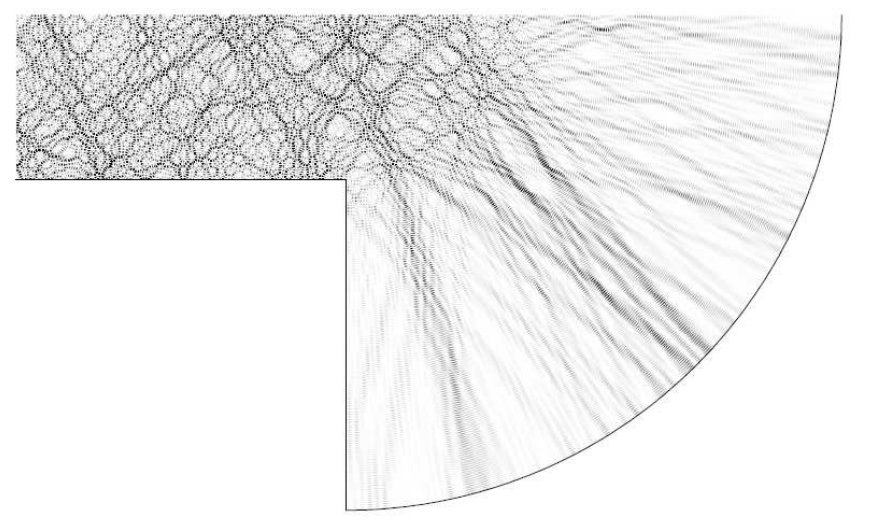}
\end{center}
\begin{center}
Figure 1: The half-mushroom billiard, with a high energy eigenfunction that extends by odd symmetry to the mushroom billiard. This particular eigenfunction appears to live in the ergodic region of phase space. Image courtesy of Dr Barnett.
\end{center}
\vspace{15pt}

This billiard, proposed by Bunimovich \cite{bunimovich} is neither classically ergodic nor completely integrable for $t >0$ and is rather one of the simplest billiards that satisfies the following mixed dynamical assumptions. 

\begin{itemize}
\label{conditions}
\item $M$ is a smooth Riemannian manifold with piecewise smooth boundary\\
\item The flow domain $\mathcal{D}$ is the union of two invariant subsets, each of positive Liouville measure and one of which, $U$, has ergodic geodesic flow\\
\item The billiard flow is completely integrable on $\mathcal{D}\setminus U$.\\
\end{itemize}

In the mushroom billiard, $U_t$ consists of $\mu_L$-almost all trajectories that enter $R_t\cup \overline{B(0,r_1)}$ before their first boundary collision. The trajectories that do not enter $R_t\cup \overline{B(0,r_1)}$ before their first boundary collision lie entirely within the upper semi-annulus $S\setminus \overline{B(0,r_1)}$ and are just reflected trajectories of the disk billiard. The integrability of the geodesic flow on $\mathcal{D}\setminus U_t$ then follows from the integrability of the disk billiard.

In the case of such mixed systems, we do not yet have a satisfactory analogue to the quantum ergodicity theorem. It is a long-standing conjecture of Percival \cite{percival} that a full density  subset of a complete system of eigenfunctions of the Laplace--Beltrami operator can be divided into two disjoint subsets, one corresponding to the ergodic region of phase space and the other corresponding to the completely integrable region. Moreover, the natural density of these subsets is conjectured to be in proportion to the Liouville measures of the corresponding flow-invariant subsets of $\mathcal{D}$. 

\begin{con}[Percival's Conjecture]
\label{percival}
For every compact Riemannian manifold $M$ such that $\mathcal{D}$ is the disjoint union of two invariant subsets $U,\mathcal{D}\setminus U$, with $U$ ergodic and $\mathcal{D}\setminus U$ integrable, we can find two subsets $A,B\subset \N$ such that

\begin{enumerate}
\item $A\cup B$ has density $1$ \\
\item $(u_k)_{k\in A}$ equidistributes in the ergodic region $U$ \\
\item Each semiclassical measure associated to the subset $B$ is supported in the completely integrable region $\mathcal{D} \setminus U$\\
\item The density of $A$ is equal to $\mu_L(U)$.\\
\end{enumerate}
\end{con}

Numerical evidence due to Barnett-Betcke \cite{barnett} has strongly supported this conjecture for the mushroom billiard. In this paper, we prove Conjecture \ref{percival} is indeed true for the mushroom billiard, at least for almost all $t\in (0,2]$. Essential in our work is the following result due to Galkowski \cite{galkowski}.

\begin{thm}
\label{galkowski}
For any compact Riemannian manifold with boundary satisfying \eqref{conditions}, there exists a full density subsequence of $(u_j)$, such that every associated semiclassical measure $\mu$ satisfies
\begin{equation}
\mu|_U = a\mu_L|_U
\end{equation}
for some constant $a$.
\end{thm}

Our strategy for this proof is motivated by that used by Hassell in constructing the first known example of a non-QUE ergodic billiard \cite{hassell}. \\

We begin in Section 2 by using the Dirichlet eigenfunctions on the semicircle to construct a family $(v_n,\alpha_n^2)$ of $O(n^{-\infty})$ quasimodes that are almost orthogonal and are microlocally supported in the completely integrable region $S^*M_t\setminus U_t$.

Using the well-known asympotics of the Bessel function zeroes, we obtain a lower bound \eqref{quasiasymptoticeqn} for the counting function of this quasimode family.\\

In Section 3, the main result is Proposition \ref{spectral}, an abstract spectral theoretic result that allows us to approximate certain  eigenfunctions by linear combinations of quasimodes of similar energy given that the numbers of each are comparable. This is the essential ingredient for passing from localisation properties about our explicit family of quasimodes to localisation properties of a family of eigenfunctions with asymptotically equivalent counting function.\\

In Section 4 we commence our study of the variation of eigenvalues as the stalk length $t$ varies in $(0,2]$. In order to simplify the nomenclature, we often interpret $t$ as a time parameter. 

The Hadamard variational formula asserts that
\begin{equation}
\label{classhad}
\dot{E}(t)=-\int_{\partial M_t} \rho_t(s)(d_n u(t)(s))^2 \, ds
\end{equation}
where $\rho_t(s)$ is the unit normal variation of the domain at a boundary point $s$. 
For normally expanding domains such as ours, \eqref{classhad} directly implies that individual eigenvalues are non-increasing in $t$.

However, using an interior formulation of the Hadamard variational formula from Proposition \ref{hadamard}, we can also quantify the variation of the  eigenvalue $E_j(t)$ by
\begin{equation}
E_j^{-1}(t)\dot{E_j}(t)=\langle Qu_j(t),u_j(t)\rangle
\end{equation}
for an appropriate pseudodifferential operator $Q$ supported in the stalk $R_t\subset M_t$. 

Proposition \ref{heatkernelprop} then establishes that for a full density subset of the eigenvalues, the quantity $\langle Qu_j(t),u_j(t)\rangle$ can be approximated up to an error of $o(E_j)$ by cutting off $Q$ sufficiently close to the boundary $\partial M_t$. This result is shown by using analysis of the wave kernel to establish the key spectral projector estimates \eqref{projest} and \eqref{gradprojest}. 

We can then use the equidistribution result of Galkowski's Theorem \ref{galkowski} to asymptotically  control $\langle Qu_j(t),u_j(t)\rangle$ and hence provide us with an upper bound \eqref{flowspeedeq} on the speed of eigenvalue variation for almost all eigenvalues.\\

Section 5 completes the argument in two parts. 

In the first of these parts, we define a set $\mathcal{G}\subset (0,2]$ such that for $t\in\mathcal{G}$, we have a certain spectral non-concentration property on $M_t$. Precisely, we have that 
\begin{equation}\begin{gathered}
\text{
the number of eigenvalues lying in the union $\cup_{j=1}^n [\alpha_j^2-c,\alpha_j^2+c]$} \\ 
\text{can exceed $n$ by at most a small proportion, for large $n$.}
\end{gathered}\label{snc}\end{equation}

Proposition \ref{spectral} then implies that these eigenfunctions are asymptotically well-approximated by linear combinations of the previously constructed family of quasimodes $(v_n)$ which are microlocally supported in the completely integrable region $S^*M_t\setminus U_t$ of phase space.


In fact, the explicit computation \eqref{quasiasymptoticeqn} of the counting function of these quasimodes leads to a proof that the corresponding family of eigenfunctions must fill up phase space. We show this is Theorem \ref{mainthm}.


Consequently, we show in Proposition \ref{ergodicprop} that a full density subset of the complementary family of eigenfunctions must have all semiclassical mass in the ergodic region $U_t$. From Theorem \ref{galkowski}, this family must then equidistribute in $U_t$ as required.\\

The final part of the paper establishes via contradiction that $(0,2]\setminus \mathcal{G}$ is Lebesgue-null. 
As in \cite{hassell} we can choose the eigenvalue branches $E_j(t)$ to be in increasing order and piecewise smooth in $t$. The crucial ingredient here is then the asymptotic bound \eqref{flowspeedeq} on the speed of eigenvalue variation. 

If $\mathcal{G}$ is not of full measure, we can construct a small interval $\mathcal{I}=[t_1,t_2]$ in which the average number of eigenvalues $E_j(t)$ lingering near quasi-eigenvalues $\alpha_i^2$ exceeds $d=d(t_1)=1-\mu_L(U_{t_1})$ by using the negation of \eqref{snc}. 

Now Weyl's law
\begin{equation}
N_t(\lambda^2)\sim \frac{\lambda^2 |M_t|}{4\pi}
\end{equation}
implies that the decrease of eigenvalues over $\mathcal{I}$ is asymptotically given by
\begin{equation}
\label{weyl}
E_j(t_1)-E_j(t_2)\sim 4\pi j (A(t_1)^{-1}-A(t_2)^{-1})
\end{equation}
in $\mathcal{I}$ as $j\rightarrow \infty$.

We can use \eqref{weyl} together with the fact that the small windows about quasi-eigenvalues are comparatively sparse in the interval $[E_j(t_2),E_j(t_1)]$ to show that the upper bound \eqref{flowspeedeq} on eigenvalue speed provides a lower bound of $(1-d)$ on the time they must spend travelling outside of quasi-eigenvalue windows. 

This implies that the average proportion of time spent by large eigenvalues lingering near quasi-eigenvalues for $t\in\mathcal{I}$ cannot exceed $d$, and consequently that the proportion of lingering eigenvalues cannot exceed $d$. This contradiction concludes the proof.\\

I would like to thank my doctoral supervisor Professor Hassell for suggesting this problem and for our many fruitful discussions regarding it.

\section{Quasimodes}

In polar coordinates, the Dirichlet eigenfunctions for the semidisk are given by
\begin{equation}
u_{n,k}:=\sin(n\theta)J_n(\alpha_{n,k}r/r_2)
\end{equation}
where $\alpha_{n,k}$ is the $k$-th positive zero of the $n$-th order Bessel function $J_n$.

\begin{prop}
If we define
\begin{equation}
v_{n,k}:=\frac{\chi(r)u_{n,k}}{\|\chi u_{n,k}\|_{L^2}},
\end{equation}
where
\begin{equation}
\chi(r)=\begin{cases}
0 &\mbox{for }r \leq r_1\\
1 &\mbox{for }r\geq (r_1+\epsilon)\sqrt{1-\epsilon^2}>r_1.
\end{cases}
\end{equation}
then the family
\begin{equation}
\label{family}
\{(v_{n,k},\alpha_{n,k}^2/r_2^2):\alpha_{n,k}<\frac{nr_2}{r_1+\epsilon}\}
\end{equation}
forms an $O(n^{-\infty})$ family of quasimodes, with all semiclassical mass contained in the completely integrable region $S^{*}M_t\setminus U_t$ of the billiard.

Moreover, these quasimodes are almost orthogonal, in the sense that
\begin{equation}
|\langle v_{n,k},v_{m,l}\rangle |=O(\min(n,m)^{-\infty})=O(\min(\alpha_{n,k},\alpha_{m,l})^{-\infty}).
\end{equation}
\end{prop}
\begin{proof}

The restriction on $k$ in our family implies that the error incurred in cutting off only depends on the values of the Bessel function $J_n(x)$ for $x\in[0,n\sqrt{1-\epsilon^2}]$.

Then from \cite{absteg}, we have the estimates
\begin{equation}
|J_n(nx)|\leq \frac{x^ne^{\sqrt{1-x^2}}}{(1+\sqrt{1-x^2})^n}\quad \textrm{ for $x\leq 1$}
\end{equation}
and
\begin{equation}
|J_n'(nx)|\leq \frac{(1+x^2)^{1/4}x^ne^{\sqrt{1-x^2}}}{x\sqrt{2\pi n}(1+\sqrt{1-x^2})^n}\quad \textrm{ for $x\leq 1$}
\end{equation}
for bounding the Bessel function near $0$.

Together these estimates imply that the error incurred by cutting off is $O(n^{-\infty})$. Furthermore, as the $u_{n,k}$ are pairwise orthogonal, these bounds also show that the $v_{n,k}$ are almost orthogonal in the sense claimed.\\

Now for any smooth compactly supported symbol $a$ spatially supported in $R_t\cup B(0,r_1)$, the disjointness of supports from our family of quasimodes implies that
\begin{equation}
\langle a(x,(r_2/\alpha_{n,k})D)v_{n,k},v_{n,k}\rangle=O(n^{-\infty}).
\end{equation}
In particular, we have that any semiclassical measure $\mu$ associated to these quasimodes cannot have mass in the region $\{(x,\xi)\in S^*M_t:x\in R_t\cup B(0,r_1)\}$.

Moreover, by the flow invariance of semiclassical measures (See Theorem 5.4 in \cite{zworski}), this implies that any corresponding semiclassical measure cannot have mass in the ergodic region $U_t$ because the pre-images under geodesic flow of \\ \noindent ${\{(x,\xi)\in \mathcal{D}_t:x\in R_t\cup B(0,r_1)\}}$ cover $U_t$.
\end{proof}

\begin{prop}
\label{decay}
We can index these quasimodes as $(v_n,\alpha_n^2)$ so that the quasi-eigenvalues are in increasing order, whilst having
\begin{equation}
(\Delta+\alpha_n^2)v_n=O(n^{-\infty})=O(\alpha_n^{-\infty})
\end{equation}
and
\begin{equation}
|\langle v_n,v_k\rangle|=O(\min(n,k)^{-\infty})=O(\min(\alpha_n,\alpha_k)^{-\infty}).
\end{equation}
\end{prop}

Moreover, as $\epsilon\rightarrow 0$, the counting function of these quasimodes has the following asymptotic bound.

\begin{prop}
\label{quasiasymptotic}
\begin{equation}
\label{quasiasymptoticeqn}
\liminf_{\lambda}\frac{\#\{(n,k):\alpha_{n,k}/r_2 < \lambda,\alpha_{n,k}<\frac{nr_2}{r_1+\epsilon}\}}{\lambda^2}\geq \left(1-\frac{\mu_L(U_t)}{\mu_L(\mathcal{D}_t)}+o(1)\right)\cdot \frac{A(t)}{4\pi}.
\end{equation}
where $A(t)$ is the area of the mushroom $M_t$.
\end{prop}
\begin{proof}
To simplify our calculations, we scale $\mu_L$ so that $\mu_L(M_t)=2\pi A(t)$.

From an arbitrary point $(r,\theta)$ in the annulus, the trajectories that never enter the stalk have measure $(2\pi-4\sin^{-1}(r_1/r))$ out of the full measure $2\pi$ of the unit cosphere at that point.

Hence

\begin{eqnarray*}
\mu_L(\mathcal{D}_t)-\mu_L(U_t)&=&\int_0^\pi\int_{r_1}^{r_2} r(2\pi-4\sin^{-1}(r_1/r))\, dr \, d\theta\\
&=& \pi^2(r_2^2-r_1^2)-4\pi\int_{r_1}^{r_2} r\sin^{-1}(r_1/r)\, dr \\
&=& \pi^2r_2^2-2\pi r_1^2\sqrt{C^2-1}-2\pi r_2^2\sin^{-1}(C^{-1})
\end{eqnarray*}
where $C=r_2/r_1$.

This implies that 
\begin{equation}
\left(1-\frac{\mu_L(U_t)}{\mu_L(\mathcal{D}_t)}\right)\cdot \frac{A(t) \lambda^2}{4\pi}=\frac{r_2^2}{8}\left(1-\frac{2}{\pi C^2}\sqrt{C^2-1}-\frac{2}{\pi}\sin^{-1}(C^{-1})\right)\lambda^2.
\end{equation}

To estimate the left hand side of \eqref{quasiasymptoticeqn}, we use the leading order uniform asymptotics for Bessel function zeros found in \cite{absteg}. 

As $n\rightarrow \infty$, we uniformly have 
\begin{equation}
\alpha_{n,k}=nz(n^{-2/3}a_k)+o(n)
\end{equation}

where $z:(-\infty,0]\rightarrow [1,\infty)$ is defined implicitly by 

\begin{equation}\frac{2}{3}(-\zeta)^{3/2}=\sqrt{z(\zeta)^2-1}-\sec^{-1}(z(\zeta))\end{equation}

and the $a_k$ are the negative zeros of the Airy function, which have asymptotic
\begin{equation}
a_k=-\left(\frac{3\pi k}{2}\right)^{2/3}+O(k^{-1/3}).
\end{equation} 

We now write $C_\epsilon=r_2/(r_1+\epsilon)$.

We count the left hand side of \eqref{quasiasymptoticeqn} by separating into two regimes based on the size of $n/\lambda$. In each of these two regimes, a single one of the inequalities defining our family \eqref{family} implies the other. More precisely, we have
\begin{eqnarray}\nonumber & &|\{(n,k):\alpha_{n,k}/r_2\leq\lambda,\alpha_{n,k}\leq C_\epsilon n\}|\\
\nonumber &= &|\{(n,k):n\leq r_2\lambda/C_\epsilon,\alpha_{n,k}\leq C_\epsilon n\}|\\ 
\nonumber &+&|\{(n,k):r_2\lambda/C_\epsilon < n \leq r_2\lambda,\alpha_{n,k}/r_2 \leq \lambda\}|\\
\label{regime} &=:&N_A(\lambda;\epsilon)+N_B(\lambda;\epsilon).\end{eqnarray}

For $n,k$ sufficiently large, a sufficient condition for being in regime $A$ of \eqref{regime} is to have
\begin{equation}
z(n^{-2/3}a_k)\leq C_\epsilon-\epsilon=\hat{C_\epsilon}
\end{equation}
and
\begin{equation}
n\leq r_2\lambda/C_\epsilon.
\end{equation}

Also, from the Airy function asymptotics we have
\begin{equation}
\frac{2}{3}(-n^{-2/3}a_k)^{3/2}=\frac{2}{3n}\left(\left(\frac{3\pi k}{2}\right)^{2/3}+O(k^{-1/3})\right)^{3/2}=\frac{\pi k}{n}+O(n^{-1}).
\end{equation}

Hence from the monotonicity of $z$, for all $n,k$ sufficiently large with $n\leq r_2\lambda/C_\epsilon$, a sufficient condition for being in regime $A$ of \eqref{regime} is
\begin{equation}
k \leq \frac{\sqrt{\hat{C_\epsilon}^2-1}-\sec^{-1}(\hat{C_\epsilon})-\epsilon}{\pi}n.
\end{equation}

Noting that the contribution from small $n$ and $k$ is finite, we can conclude that
\begin{eqnarray*}
\liminf_\lambda\frac{N_A(\lambda^2)}{\lambda^2}&\geq &\liminf_\lambda (\frac{1}{\lambda^2}\sum_{n\leq r_2\lambda/C_\epsilon} n)\cdot \frac{\sqrt{\hat{C_\epsilon}^2-1}-\sec^{-1}(\hat{C_\epsilon})-\epsilon}{\pi}\\
&=& \frac{r_2^2(\sqrt{\hat{C_\epsilon}^2-1}-\sec^{-1}(\hat{C_\epsilon})-\epsilon)}{2C_\epsilon^2\pi}.
\end{eqnarray*}

Similarly, for sufficiently large $n,k$, a sufficient condition for being in regime $B$ of \eqref{regime} is to have
\begin{equation}
z(n^{-2/3}a_k)\leq \lambda r_2/n -\epsilon=D_\epsilon(\lambda,n).
\end{equation}
and
\begin{equation}
r_2\lambda/C_\epsilon < n \leq r_2\lambda.
\end{equation}

Hence, for all $n,k$ sufficiently large with $r_2\lambda/C_\epsilon < n \leq r_2\lambda/(1+\epsilon)$, a sufficient condition for being in regime $B$ of \eqref{regime} is
\begin{equation}
k\leq \frac{\sqrt{D_\epsilon(\lambda,n)^2-1}-\sec^{-1}(D_\epsilon(\lambda,n))-\epsilon}{\pi}n.
\end{equation}

Again throwing away a finite number of small pairs, we obtain
\begin{eqnarray*}
& & \liminf_\lambda \frac{N_B(\lambda^2)}{\lambda^2}\\
&\geq & \liminf_\lambda \frac{1}{\pi\lambda^2} \sum_{\frac{r_2\lambda}{C_\epsilon} < n \leq \frac{r_2\lambda}{1+\epsilon}}\left(n\sqrt{D_\epsilon(\lambda,n)^2-1}-  n\sec^{-1}(D_\epsilon(\lambda,n)) -\epsilon n\right)\\
&=& \liminf_\lambda\frac{1}{\pi\lambda^2}\int_{\frac{r_2\lambda}{C_\epsilon}}^{\frac{r_2\lambda}{1+\epsilon}} \left(t\sqrt{D_\epsilon(\lambda,t)^2-1} -  t\sec^{-1}(D_\epsilon(\lambda,t))-\epsilon t\right)\, dt.
\end{eqnarray*}
Each of the three summands in the integrand has elementary primitive, so we can explicitly compute this quantity.

Noting that $C_\epsilon,\hat{C_\epsilon}\rightarrow C$, we compute
\begin{eqnarray*}
& &\lim_{\epsilon\rightarrow 0}\liminf_\lambda \frac{N_A(\lambda^2)+N_B(\lambda^2)}{\lambda^2}\\
&\geq & \left(\frac{r_1^2\sqrt{C^2-1}}{2\pi}-\frac{r_1^2}{2\pi}(\frac{\pi}{2}-\sin^{-1}(C^{-1}))\right)\\
&+& \frac{1}{\pi \lambda^2}\int_{r_1\lambda}^{r_2\lambda} \sqrt{\lambda^2 r_2^2-t^2}\, dt-\frac{1}{\pi\lambda^2}\int_{r_1\lambda}^{r_2\lambda}t\sec^{-1}(\frac{\lambda r_2}{t})\, dt\\
&=& \frac{r_2^2}{8}\left(1-\frac{2}{\pi C^2}\sqrt{C^2-1}-\frac{2}{\pi}\sin^{-1}(C^{-1})\right).\\
\end{eqnarray*}
as required.

\end{proof}

\section{Spectral Theory}

We next establish the following key spectral theoretic result.

\begin{prop}
\label{spectral}
Let $\mathcal{H}$ be a Hilbert space. Suppose $T\in\mathcal{L}(\mathcal{H})$ has a complete orthonormal system of eigenvectors $(u_i,E_i)_{i\in\N}$ with the sequence $(E_i)$ non-negative and increasing without bound. Suppose further that we have a family of normalised quasimodes $(v_i,E_i')_{i=1}^n$ with 
\begin{equation}\label{epsilonone}\|(T-E_i')v_i\| < \epsilon_1\end{equation}
and
\begin{equation}\label{epsilontwo}|\langle v_i,v_j\rangle| < \epsilon_2 \quad \textrm{for }i\neq j\end{equation}
for some positive $\epsilon_1,\epsilon_2>0.$

We write 
\begin{equation}
V=\textrm{Span}\{v_i\}_{i=1}^n
\end{equation}
and
\begin{equation}
U=\textrm{Span}\{u_j:E_j\in\bigcup_{i=1}^n [E_i'-c,E_i'+c]\}.
\end{equation}

We denote the orthogonal projection onto a subspace $S\subseteq \mathcal{H}$ by $\pi_S$.

If for some $c>0$ and some $0<\epsilon,\delta<1/2$ we have

\begin{equation}
m=\#\{j\in\N:E_j\in\bigcup_{i=1}^n [E_i'-c,E_i'+c] \}<n(1+\epsilon)
\end{equation}
and
\begin{equation}
\label{errormatrixbound}
\frac{\epsilon_1^2}{c^2}+\epsilon_2 < \frac{\delta}{n}
\end{equation}

then at least $n(1-\sqrt{\epsilon})$ of the corresponding eigenvectors $u_i$ satisfy

\begin{equation}
\label{eigenestimate}
\|u_i-\pi_V(u_i)\|<\epsilon^{1/4}+2\delta^{3/2}. 
\end{equation}

\end{prop}
\begin{proof}
The idea behind the proof of the estimate \eqref{eigenestimate} consists of several successive approximations. 

We first show that the projections $\pi_U(v_i)$ are almost orthogonal and can be transformed into an orthonormal basis $(w_i)_{i=1}^n$ of their span by a matrix $A$ that is approximately the identity. 

We then show that excluding some exceptional eigenvectors, the remaining eigenvectors are necessarily rather close to the space $W$. This implies that the non-exceptional eigenvectors can be well approximated by their projections, which leads us to conclude $u\approx\pi_W(u)=Bw=BA\pi_U(v)\approx BAv$ for some matrix $B$.\\

To begin, we reindex the eigenpairs $(u_i,E_i)$ so that $E_j\in \cup_{i=1}^n [E_i'-c,E_i'+c]$ precisely for $j=1,2,\ldots,m.$

The assumptions \eqref{epsilonone} and \eqref{epsilontwo} then imply

\begin{eqnarray*}
\|(T-E_i')\sum_{j\in\mathbb{N}}\langle v_i,u_j\rangle u_j\|^2 & < & \epsilon_1^2\\ 
\Rightarrow \sum_{j=m+1}^\infty |E_j-E_i'|^2|\langle v_i,u_j\rangle|^2 &<& \epsilon_1^2\\
\Rightarrow \sum_{j=m+1}^\infty |\langle v_i,u_j\rangle|^2 &<& \frac{\epsilon_1^2}{c^2}\\
\Rightarrow \|\pi_U(v_i)\|^2 &>& 1-\frac{\epsilon_1^2}{c^2}.
\end{eqnarray*}
and
\begin{eqnarray*}
|\langle\pi_U(v_i),\pi_U(v_j)\rangle|&\leq& |\langle v_i,v_j\rangle|+|\langle\pi_{U^\perp}(v_i),\pi_{U^\perp}(v_j)\rangle|\\
& < & \epsilon_2 + \sqrt{(1-\|\pi_U(v_i)\|^2)(1-\|\pi_U(v_j)\|^2)}\\
& < & \epsilon_2+\frac{\epsilon_1^2}{c^2}
\end{eqnarray*}
for $i\neq j$.

Together with \eqref{errormatrixbound} we obtain
\begin{equation}
\label{ineqone}
\|\pi_U(v_i)\|^2 > 1-\frac{\delta}{n}
\end{equation}
and
\begin{equation}
\label{ineqtwo}
|\langle \pi_U(v_i),\pi_U(v_j)\rangle|< \frac{\delta}{n}
\end{equation}
for $i\neq j$.

The matrix $M$ with entries $M_{ij}=\langle\pi_U(v_i),\pi_U(v_j)\rangle$ satisfies

\begin{equation}
\label{approxid}
\|M-I\|_{HS}=:\|E\|_{HS}<\delta<1/2.
\end{equation}

Note that if the collection  $\{\pi_U(v_i)\}$ were linearly dependent, then the matrix $M$ would be singular. The estimate \eqref{approxid} precludes this possibility, because we can invert $M=I-(I-M)$ as a Neumann series. In particular, this implies that $m\geq n$.\\

We now write $W=\textrm{Span}\{\pi_U(v_i)\}_{i=1}^n$ and suppose that $(w_i)_{i=1}^n$ is an orthonormal basis for $W$ which can be given by the transformation $w=A\pi_U(v)$, where $A$ is an $n\times n$ real matrix that acts on the Hilbert space $\mathcal{H}^n$ via matrix multiplication.

Expanding out the matrix equation $\langle w_i,w_j\rangle=\delta_{ij}$ we obtain
\begin{equation}
AMA^*=\sum_{k=1}^n\sum_{l=1}^na_{ik}\overline{a_{jl}}\langle \pi_U(v_k),\pi_U(v_l)\rangle=I
\end{equation}
which has a solution
\begin{equation}
A=M^{-1/2}=(I+E)^{-1/2}=\sum_{k=0}^\infty \binom{-1/2}{k}E^k=\sum_{k=0}^\infty (-1)^k \binom{2k}{k}4^{-k}E^k.
\end{equation}

From \eqref{approxid}, we then deduce
\begin{equation}
\|A-I\|_{HS}\leq \sum_{k=1}^\infty \|E\|^k = \|E\|(1-\|E\|)^{-1}<2\delta.
\end{equation}

In the case $m=n$, that is when $W=U$, we can find a unitary matrix $B$ with $Bw=u$.

We now assume $m>n$, recalling that the assumptions of the proposition imply that this excess is small as a proportion of $n$.

We have
\begin{eqnarray*}
\sum_{i=1}^m \|\pi_{W^\perp}(u_i)\|^2=m-\sum_{i=1}^m \|\pi_{W}(u_i)\|^2=m-\sum_{i=1}^m\sum_{j=1}^n |\langle u_i,w_j\rangle|^2=m-n<n\epsilon
\end{eqnarray*}

which implies that
\begin{equation}
\#\{i:\|\pi_{W^\perp}(u_i)\|^2\geq\sqrt{\epsilon}\}<n\sqrt{\epsilon}
\end{equation}
and consequently
\begin{equation}
\label{closest}
\#\{i:\|\pi_{W}(u_i)\|^2>1-\sqrt{\epsilon}\}\geq m-n\sqrt{\epsilon}>n(1-\sqrt{\epsilon}).
\end{equation}

We again re-index the eigenpairs for convenience, so that the first $n'=\lceil n(1-\sqrt{\epsilon}) \rceil$ eigenvectors $u_i$ satisfy the estimate in \eqref{closest}.

In this case, we define the $n' \times n$ matrix $B$ to have entries
\begin{equation}
B_{ij}=\langle\pi_W(u_i),w_j\rangle
\end{equation}
and the vector $u\in\mathcal{H}^{n'}$ by
\begin{equation}
(u_i)_{i=1}^{n'}.
\end{equation}

We then have
\begin{equation}
\pi_W(u)=Bw
\end{equation}
and the $i$-th row $B_i$ of $B$ has $\ell^2$ norm trivially bounded by $1$.

This leaves us with
\begin{equation}
u=(u-\pi_W(u))+BA\pi_U(v)=(u-\pi_W(u))+BAv+BA(\pi_U(v)-v).
\end{equation}

which implies
\begin{eqnarray*}
\|(u-BAv)_i\|_{\mathcal{H}}&\leq& \|u_i-\pi_W(u_i)\|_{\mathcal{H}}+\|B_i\|_{\ell^2}\|A\|_{HS}\|\pi_U(v)-v\|_{\mathcal{H}^n}\\
&<& \epsilon^{1/4}+\cdot (2\delta) \cdot \sqrt{\delta}\\
&<& \epsilon^{1/4}+2\delta^{3/2}.
\end{eqnarray*}

This estimate shows that each $u_i$ has distance less than $\epsilon^{1/4}+2\delta^{3/2}$ to some element of $V$. 

Consequently 
\begin{equation} 
\|u_i-\pi_V(u_i)\| < \epsilon^{1/4}+2\delta^{3/2} 
\end{equation}
as required.
\end{proof}

Our strategy to prove the main theorem is to control the number of eigenvalues in most clusters formed by finite unions of overlapping intervals of the form ${[\alpha_i^2-c,\alpha_i^2+c]}$, and then repeatedly employ Proposition \ref{spectral} to establish the existence of a large density subsequence of these eigenfunctions that localises in the semidisk.

\section{Results on Eigenvalue Flow}

Central to the argument is the analysis of how eigenvalues flow as we vary $t$. 

Weyl's law provides us with the asymptotic
\begin{equation}
N_t(\lambda^2)\sim \frac{\lambda^2|M_t|}{4\pi}
\end{equation}

where $N_t$ is the counting function of the Dirichlet eigenvalues on $M_t$.

To obtain a more precise statement about the change of individual eigenvalues, we employ an interior version of the Hadamard variational formula and Theorem \ref{galkowski}.

In order to make use of Theorem \ref{galkowski}, we choose $\phi(y)\in \mathcal{C}^\infty_c (\mathbb{\R})$ non-negative, supported near $y=-1/2$, and with integral $1$, and we define the family of metrics
\begin{equation}g_t=dx^2+(1+(t-1)\phi)^2 dy^2\end{equation}

on $M_1$.

This metric induces a natural isometry $I_t: (M_1,g_t)\rightarrow (M_t,g_1)$.

If we define $R_t=(1+(t-1)\phi)^{-1/2}$, then we have the following result from Proposition 7 of the appendix of \cite{hassell}.

\begin{prop}
\label{hadamard}
Let $u(t)$ be an $L^2$-normalised real Dirichlet eigenfunction of $\Delta$ on $M_t$ with corresponding eigenvalue $E(t)$. We then have
\begin{equation}
\dot{E}(t)=-\frac{1}{2}\langle Qu(t),u(t)\rangle
\end{equation}

where the operator $Q$ is given by
\begin{equation}
Q=-4\partial_y\phi_t\partial_y + [\partial_y,[\partial_y,\phi_t]]=\phi_t''-4(\phi_t'\partial_y+\phi_t\partial_y^2)
\end{equation}

on $M_t$.

Here, $\phi_t:M_t\rightarrow \R$ is given by:
\begin{equation}
\phi_t=(\phi R_t^2)\circ I_t^{-1}.
\end{equation}
\end{prop}

We now cut $Q$ off away from the vertical sides of the stalk so that we can use the interior equidistribution result Theorem \ref{galkowski} to control the quantity $E_k^{-1}\dot{E_k}$.

We do this by defining
\begin{equation}Q_\delta=\chi_\delta Q \end{equation}

where $\chi_\delta\in\mathcal{C}^\infty$ satisfies
\begin{equation}
\chi_\delta(x)=\begin{cases}
0 &\mbox{for }x\in [-r_1,-r_1+\delta]\cup [r_1-\delta, r_1]\\
1 &\mbox{for }x \in [-r_1+2\delta,r_1-2\delta].
\end{cases}
\end{equation}

\begin{prop}
\label{heatkernelprop}
For any $\epsilon > 0$ and any $t\in (0,2]$, there exists $\delta>0$ such that
\begin{equation}
\label{heatkernel}
|E_{n_k}^{-1}\langle (Q_\delta-Q)u_{n_k}(t),u_{n_k}(t)\rangle|<\epsilon
\end{equation}
for all $k$, where $(n_k)$ is a $t$-dependent subsequence of the positive integers with lower density bounded below by $1-\epsilon$.
\end{prop}
\begin{proof}
First we show that it suffices for each $t$ to establish the spectral projector estimates
\begin{equation}
\label{projest}
\|\eta_t 1_{[\lambda,\lambda+1)}(\sqrt{-\Delta})\|_{L^2(M_t)\rightarrow L^\infty(M_t)}=O(\lambda^{1/2})
\end{equation}

and
\begin{equation}
\label{gradprojest}
\|\eta_t \nabla 1_{[\lambda,\lambda+1)}(\sqrt{-\Delta})\|_{L^2(M_t)\rightarrow L^\infty(M_t)}=O(\lambda^{3/2}).
\end{equation}

Here $\eta_t=\eta\circ I_t^{-1}$ where $\eta:M_1\rightarrow \mathbb{R}$ is an fixed smooth cutoff function supported and equal to $1$ in a neighbourhood of ${\partial M_1 \cap \textrm{spt}(\phi)}$ such that $\eta$ vanishes in a neigbourhood of the semidisk.


Applying $\eta_t\nabla 1_{[\lambda,\lambda+1)}(\sqrt{\Delta})$ to $\sum_{\lambda_j\in[\lambda,\lambda+1)} a_j u_j$ and using the estimate \eqref{gradprojest} then yields

\begin{eqnarray*}
\left|\eta_t(x)\sum_{\lambda_j\in[\lambda,\lambda+1)}a_j\nabla u_j(x)\right|&\leq& C\lambda^{3/2}\left\|\sum_{\lambda_j\in[\lambda,\lambda+1)}a_j u_j\right\|_{L^2}\\ 
&\leq& C\lambda^{3/2}\left(\sum_{\lambda_j\in[\lambda,\lambda+1)} |a_j|^2\right)^{1/2}
\end{eqnarray*}

for each $x\in M_t$.

Setting $a_j=\nabla u_j(x)$ then yields the estimate
\begin{equation}
\label{projest2}
\eta_t(x)^2\sum_{\lambda_j\in[\lambda,\lambda+1)} |\nabla u_j(x)|^2\leq C \lambda^3.
\end{equation}

Similarly, we obtain
\begin{equation}
\label{gradprojest2}
\eta_t(x)^2\sum_{\lambda_j\in[\lambda,\lambda+1)} |u_j(x)|^2\leq C \lambda.
\end{equation}

The estimates \eqref{projest2} and \eqref{gradprojest2} then allow us to control each term of \eqref{heatkernel} in an average sense.

For example, as only the the horizontal component $1-\chi_\delta$ of the cutoff function in \eqref{heatkernel} is $\delta$-dependent, we can integrate by parts in the second order term in \eqref{heatkernel} without loss. 

Then, by writing $\eta_{\delta}$ to denote the cutoff function in the new second order term and choosing $\delta >0$ sufficiently small so that $\eta_\delta=\eta\eta_\delta$, the contribution of these terms to \eqref{heatkernel} is controlled by
\begin{eqnarray*}
& &E^{-1}\sum_{E_j\leq E} E_j^{-1}\int_M \eta_\delta(x)|\nabla u_j(x)|^2 \, dx\\
&\sim &E^{-1}\sum_{k=1}^{E^{1/2}-1}\sum_{\lambda_j\in [k,k+1)} E_j^{-1}\int_M \eta_\delta(x)|\nabla u_j(x)|^2 \, dx\\
&\leq & E^{-1}\sum_{k=1}^{E^{1/2}-1}k^{-2} \int_M \eta_\delta(x)\eta(x)\sum_{\lambda_j\in [k,k+1)}|\nabla u_j(x)|^2 \, dx\\
& \leq & CE^{-1}\left(\sum_{k=1}^{E^{1/2}-1} k\right) \int_M \eta_\delta(x)\, dx\\
&\leq & \hat{C_\delta}
\end{eqnarray*}
for sufficiently small $\delta>0$, where $\hat{C_\delta}\rightarrow 0$ as $\delta\rightarrow 0$.

Together with analogous estimates for lower order terms, we obtain the estimate
\begin{equation}
\frac{1}{n}\sum_{j=1}^n E_j^{-1}|\langle(Q_\delta-Q)u_j,u_j)\rangle|<C_\delta
\end{equation}
where $C_\delta\rightarrow 0$ as $\delta\rightarrow 0$.

By taking $\delta$ sufficiently small that $C_\delta < \epsilon^2$, we ensure that the collection of $j$ with $E_j^{-1}|\langle(Q_\delta-Q)u_j,u_j)\rangle|\geq \epsilon$ has upper density at most $\epsilon$.\\

The estimate \eqref{projest} follows from Proposition 8.1 in \cite{restrict}. Note that the finite propagation speed of the operator $\cos(t\sqrt{-\Delta})$, the post-composition with a cutoff near a flat boundary, and the small-time nature of the argument together imply that $M_t$ can be treated as the half-plane, which certainly satisfies the geometric assumptions of the cited result.

By inserting the gradient operator in the dual estimate, it remains to control  $\|1_{[\lambda,\lambda+1)}(\sqrt{-\Delta})\nabla \eta_t\|_{L^1(M_t)\rightarrow L^2(M_t)}$ in order to give us \eqref{gradprojest}. 

The argument in the proof of Proposition 8.1 in \cite{restrict} allows us to replace the spectral projector by a smooth spectral projector $\rho_\lambda^{ev}(\sqrt{-\Delta})$ where ${\rho_\lambda^{ev}(s)=\rho(s - \lambda)+\rho(-s-\lambda)}$ and $\rho\in\mathcal{S}(\R)$ has non-negative Fourier transform supported in $[\epsilon/2,\epsilon]$ for some sufficiently small $\epsilon$.

So it suffices to estimate the $L^1\rightarrow L^2$ norm of the operator
\begin{equation}
\rho_\lambda^{ev}(\sqrt{-\Delta})\nabla=\frac{1}{\pi}\int_\R \cos(t\sqrt{-\Delta})(e^{-it\lambda}\hat{\chi}(t)+e^{it\lambda}\hat{\chi}(-t))\nabla \eta(x) \, dt.
\end{equation}

This integral is supported close to $t=0$, and hence by finite propagation speed, the kernel of the wave equation solution operator $\cos(t\sqrt{-\Delta})$ on $M$ is identical to that of the half-plane.

Moreover, the kernel of the wave equation solution operator on the half-plane can be obtained from the free space wave kernel by the reflection principle, and their $L^1\rightarrow L^2$ norms are identical.

This implies that it suffices to prove the estimate with the kernel for $\cos(t\sqrt{-\Delta})$ replaced by the free space wave kernel.

So the kernel to be estimated is
\begin{eqnarray*}
& &\frac{1}{4\pi^3}\int_\R \int_{\R^2}\int_{\R^2} e^{i(x-y)\cdot \xi}(e^{-it\lambda}\hat{\chi}(t)+e^{it\lambda}\hat{\chi}(-t))\xi \cos(|\xi|t)\eta(x)\,dy \,d\xi \, dt\\
&=& \nabla_x(K_\lambda(x,y))\eta(x)\\
&=& \nabla_x(\lambda^{(n-1)/2}a_\lambda(x,y)e^{i\lambda\psi(x,y)})\eta(x)\\
&=& O(\lambda^{(n+1)/2})
\end{eqnarray*}
where $K_\lambda,a_\lambda,\psi$ are as in Lemma 5.13 from \cite{sogge}, which we make use of in our penultimate line.

Duality then completes the proof of \eqref{gradprojest} and the proposition.

\end{proof}

Proposition \ref{hadamard} allows us to use Theorem \ref{galkowski} and Proposition \ref{heatkernelprop} to control the flow speed of a full density subsequence of the eigenfunctions for any fixed $t$.

\begin{prop}
\label{flowspeed1}
For each $t\in(0,2]$ there exists a full density subsequence $(n_k)$ of the positive integers such that
\begin{equation}
\liminf_{k\rightarrow\infty} E_{n_k}^{-1}(t)\dot{E}_{n_k}(t) \geq -\frac{\dot{A}(t)}{A(t)(1-d(t))}
\end{equation}
where $d(t)$ denotes the proportion of the phase space volume that is in the completely integrable region $S^{*}M_t\setminus U_t$.
\end{prop}
\begin{proof}
From Proposition \ref{hadamard}, Proposition \ref{heatkernelprop} and Theorem \ref{galkowski}, for all $\epsilon>0$ we may choose a $\delta>0$ and a subsequence of eigenfunctions with lower density bounded below by $1-\epsilon$ such that we have the estimate \eqref{heatkernel} and such that we have a unique semiclassical measure $\mu$ with $\mu|_{U_t}=a\mu_L|_{U_t}$ for some non-negative constant $a$.
We immediately have
\begin{equation}
a=\frac{\mu(U_t)}{1-d(t)}\leq \frac{1}{1-d(t)}.
\end{equation}

The definition of semiclassical measures then implies
\begin{eqnarray*}
\liminf_{k\rightarrow\infty} E_{n_k}^{-1}(t)\dot{E}_{n_k}  &\geq & -\frac{1}{2}\int_{S^*M} \sigma(Q_\delta) \, d\mu-\epsilon\\
&\geq & -\frac{1}{2}\int_{S^*M} 4\phi_t\xi_2^2 \, d\mu-\epsilon \\
&\geq & -\frac{2}{1-d}\int_{S^*M} \phi_t\xi_2^2 \, d\mu_L-\epsilon\\
&\geq & -\frac{1}{\pi(1-d(t))A(t)}\int_M \int_0^{2\pi}\phi_t(x)\sin^2(\theta)\,d\theta\, dx -\epsilon\\
&\geq & -\frac{\dot{A}(t)}{(1-d)A(t)}-\epsilon.
\end{eqnarray*}

We can then apply Lemma \ref{densitylemma} to obtain a full density subsequence of eigenfunctions satisfying the estimate \eqref{flowspeedeq}.
\end{proof}

Moreover, we can strengthen the above to an almost-uniform result.

\begin{prop}
\label{flowspeed}
For any $\epsilon>0$, there exists a full-density subsequence $(n_k)$ of positive integers and a family of sets $B_{n_k}\subseteq (0,2]$ with $m(B_{n_k})\rightarrow 0$ such that
\begin{equation}
\label{flowspeedeq}
E_{n_k}^{-1}(t)\dot{E}_{n_k}(t) > -\frac{\dot{A}(t)}{A(t)(1-d(t))}-\epsilon
\end{equation}
for each $t\in (0,2]\setminus B_{n_k}$.
\end{prop}
\begin{proof}
For each $\delta>0$ we define the subset $S_\delta\subseteq \mathbb{N}$ as the collection of $n\in\mathbb{N}$ such that
\begin{equation}
m(\{t\in (0,2]:E_{n}^{-1}(t)\dot{E}_{n}(t) \leq -\frac{\dot{A}(t)}{A(t)(1-d(t))} -\epsilon\})>\delta.
\end{equation}

If every $S_\delta$ were of zero density, we could write $S'_\delta:=\mathbb{N}\setminus S_\delta$ and use Lemma \ref{densitylemma} to assemble a full-density set satisfying the claims of the proposition.

Now suppose that $S_\delta$ has positive upper density for some $\delta>0$. 

Since for every $n\in S_\delta$, the sets $B_n=\{t:E_{n}^{-1}(t)\dot{E}_{n}(t) \leq -\frac{\dot{A}(t)}{A(t)(1-d(t))} -\epsilon\}$ have measure bounded below and are subsets of a set with finite measure, there must exist a further positive density subset $\hat{S}_\delta\subseteq S_\delta$ such that 
\begin{equation}
\bigcap_{n\in\hat{S}_\delta} B_n\neq \emptyset.
\end{equation} 
This can be seen for instance by applying the bounded convergence theorem to the function
\begin{equation}
\frac{1}{n}\sum_{j=1}^n 1_{B_j}.
\end{equation}
The existence of a $t$ in this intersection contradicts Proposition \ref{flowspeed1} and hence completes the proof.
\end{proof}

The almost-uniform result in Proposition \ref{flowspeed} can for our purposes be treated as a uniform bound on speed for large $E_j$, in light of the following weaker bound for $t$ in the sets $B_j$ of diminishing measure for which \eqref{flowspeedeq} does not hold.

\begin{prop}
\label{almostuniform}
There exists a positive constant $C$ such that for every $t\in (0,2]$ and every $j$, we have
\begin{equation}
\dot{E}_j(t)=-\frac{1}{2}\langle Qu_j,u_j\rangle\geq -CE_j(t)
\end{equation}
\end{prop}
\begin{proof}
Integration by parts and Cauchy--Schwartz on the left-hand side provides us with a lower bound of 
\begin{eqnarray*}
& &-\hat{C} \iint_M |(\partial_y^2 u_j)u_j| +|(\partial_y u_j)u_j|+|u_j|^2 \, dx\, dy\\
&\geq & -\hat{C}(\langle -\Delta u_j,u_j \rangle+\langle -\Delta u_j, u_j \rangle^{1/2}+1 )\\
& \geq & -CE_j
\end{eqnarray*}
for a positive constant $C$ that is uniform in time.
\end{proof}

\begin{cor}
\label{speedcor}
For any $\epsilon>0$, there exists $\delta>0$ and a full-density subsequence $(n_k)$ such that we have
\begin{equation}
-\int_S \dot{E}_{n_k}(t)\, dt\leq E_{n_k}(t_1)\left(\frac{\dot{A}(t_1)}{A(t_1)(1-d(t_1))}+\epsilon\right)(t_2-t_1)
\end{equation}
for any measurable set $S\subseteq (0,2]$ with measure greater than $\delta$.
\end{cor}
\begin{proof}
This follows from Proposition \ref{flowspeed} and Proposition \ref{almostuniform} by removing finitely many elements from the subsequence constructed in Proposition \ref{flowspeed}.
\end{proof}

We are now in a position to prove the main results of the paper.

\section{Main Results}

\begin{defn}
We call $t\in(0,2]$ \emph{good} if for every $\epsilon>0$, there exists some $c>0$ with
\begin{equation}
\label{goodtime}
\limsup_{n\rightarrow\infty}\frac{\#\{j\in\N:E_j(t)\in \cup_{i=1}^n [\alpha_i^2-c,\alpha_i^2+c]\}}{n} < 1+\epsilon^2
\end{equation}
We denote the set of good times by $\mathcal{G}$.
\end{defn}

I claim that Percival's conjecture holds for the mushroom billiard $M_t$ for all $t\in\mathcal{G}$, and moreover, that this set has full measure in $(0,2]$.

First we shall prove the claim for fixed $t\in\mathcal{G}$. 

For a given $c>0$, we define \emph{$c$-clusters} to be the connected components of $\cup_{i\in\N} [\alpha_i^2-c,\alpha_i^2+c]$.

We write $N_\textrm{semidisk}(\mathcal{C}),N_\textrm{mushroom}(\mathcal{C})$ to denote the number of quasi-eigenvalues and eigenvalues respectively contained in a given $c$-cluster $\mathcal{C}$. 

The assumption \eqref{goodtime} then implies the following Proposition.

\begin{prop}
\label{throwaway}
Suppose $t$ and $0<c<2/r_2^2$ are such that \eqref{goodtime} holds, and all but finitely many $c$-clusters contain at least as many eigenvalues as quasi-eigenvalues. Then there exists a subset $J$ of quasi-eigenvalues with lower density at least $(1-\epsilon)$ such that each quasi-eigenvalue in $J$ is contained in a cluster $\mathcal{C}$ with 
\begin{equation}
\label{throwawayeq}
N_\textrm{semidisk}(\mathcal{C})\leq N_\textrm{mushroom}(\mathcal{C}) \leq (1+\epsilon)N_\textrm{semidisk}(\mathcal{C})
\end{equation}
\end{prop}
\begin{proof}
Index the $c$-clusters $\mathcal{C}_k$ in increasing order. 

Let $$S=\{k:N_\textrm{semidisk}(\mathcal{C}_k)\leq N_\textrm{mushroom}(\mathcal{C}_k)\leq(1+\epsilon)N_\textrm{semidisk}(\mathcal{C}_k)\}$$ and $$F=\{k:N_\textrm{mushroom}(\mathcal{C}_k)<N_\textrm{semidisk}(\mathcal{C}_k)\}.$$ 
From the defining property \eqref{goodtime} of $t\in\mathcal{G}$, we have:
\begin{equation}
\limsup_{n\rightarrow\infty}\frac{\sum_{k\leq n} N_\textrm{mushroom}(k)}{\sum_{k\leq n}N_\textrm{semidisk}(k)} < 1+\epsilon^2
\end{equation}
The definition of $S$ implies
\begin{multline*}
\limsup_{n\rightarrow\infty}\left(1-\frac{\sum_{k\leq n} N_\textrm{semidisk}(k)1_F(k)}{\sum_{k\leq n} N_\textrm{semidisk}(k)}\right.\\
+\left.\epsilon\frac{\sum_{k\leq n} N_\textrm{semidisk}(k)(1-1_F(k)-1_S(k))}{\sum_{k\leq n} N_\textrm{semidisk}(k)}\right)< 1+\epsilon^2.
\end{multline*}
The second term on the left-hand side is $o(1)$ from the finiteness assumption. Hence we obtain that the upper density of $\N\setminus(S\cup F)$ is bounded above by $\epsilon$. As $F$ is finite, and consequently of density $0$, we can conclude that the lower density of $S$ is bounded below by $1-\epsilon$ as required. 
\end{proof}

\begin{thm}[Main Theorem]
\label{mainthm}
For each $t\in\mathcal{G}$, there exists $B_t\subset \N$ of density $d(t)$ such that any semiclassical measure associated to the eigenfunctions $(u_n)_{n\in B_t}$ is supported inside the completely integrable region.
\end{thm}
\begin{proof}

First, we fix $\epsilon >0$ and choose $c>0$ small enough so that the inequality \eqref{goodtime} holds.

Proposition \ref{decay} implies that we may choose the $\epsilon_1,\epsilon_2$ in  applications of Proposition \ref{spectral} to the increasing sequence of $c$-clusters to decay faster than any polynomial in the energy infima of these $c$-clusters.

By Weyl's law, this ensures that for all but possibly finitely many $c$-clusters, we have \eqref{errormatrixbound} with $\delta$ decaying faster than any polynomial in energy. We remove the exceptional $c$-clusters, without any loss in density of our subset.\\

In light of Proposition \ref{throwaway}, we can then select a subset of the remaining $c$-clusters such that the included subset of quasi-eigenvalues has lower density exceeding $1-\epsilon$ and such that \eqref{throwawayeq} holds for each cluster.

We can now apply Proposition \ref{spectral} on a cluster-by-cluster basis, with parameter $\delta\rightarrow 0$ faster than any polynomial in energy.

From the $L^2$ boundedness of pseudodifferential operators with compactly supported symbols, this implies that for each $\epsilon$, we get a subsequence of eigenfunctions $u_{j_k}$ such that any associated semiclassical measure $\mu$ satisfies
\begin{equation}
\mu(\mathcal{D}_t\setminus U_t)\geq 1-\epsilon^{1/4}.
\end{equation}

Moreover, by comparison to Proposition \ref{quasiasymptotic} and Weyl's law for the mushroom, we obtain a lower bound of $d(t)-h(\epsilon)$ for the lower density of this eigenfunction subsequence with  $h(\epsilon)\rightarrow 0$ as $\epsilon \rightarrow 0$. So for each $\epsilon>0$, we can find a $c$ and a subsequence of $(u_n)$ with density at least $d(t)-h(\epsilon)$ which concentrates in the completely integrable region up to $\epsilon^{1/4}$ of its semiclassical mass.

We now take a sequence $\epsilon_j\rightarrow 0$ and denote the corresponding eigenvalue window widths by $c_j$. We write $B_{j,t}$ to denote the corresponding  concentrating eigenfunction subsequences.

Lemma \ref{densitylemma} then allows us to obtain a subsequence $B_t$ of lower density at least $d(t)$ such that any associated semiclassical measure $\mu$ satisfies
\begin{equation}
\label{local}
\mu(\mathcal{D}_t\setminus U_t)= 1.
\end{equation}

To bound the upper density of $B_t$, we choose a function $\chi_\epsilon\in\mathcal{C}_c^{\infty}(\R^2\times \R^2)$ supported in the interior of $M$ such that the following properties are satisfied.
\begin{itemize}
\item $0\leq \chi_\epsilon \leq 1$\\
\item $\chi_\epsilon | _{\mathcal{D}_t\setminus U_t}=0$\\
\item $\int_{\mathcal{D}_t} \chi_\epsilon \, d\mu_L > (1-\epsilon)\mu_L(U_t)$.
\end{itemize}

Applying the local Weyl law (Lemma 4 from \cite{zelditch-zworski}) to the corresponding semiclassical  pseudodifferential operator $\chi_\epsilon(x,hD)$, we obtain

\begin{equation}
\frac{1}{n}\sum_{j\in [1,n]\cap B_t} \langle \chi_\epsilon(x,E_j^{-1/2}D) u_j,u_j \rangle+\frac{1}{n}\sum_{j\in [1,n]\cap B_t^c}\langle \chi_\epsilon(x,E_j^{-1/2}D) u_j,u_j \rangle > (1-\epsilon)\mu_L(U_t) \end{equation}

for sufficiently large $n$.

The localisation property \eqref{local} implies that the first summand is $o(1)$ in $n$. Hence we have
\begin{equation}
\frac{1}{n}\sum_{j\in [1,n]\cap B_t^c}\langle \chi_\epsilon(x,E_j^{-1/2}D) u_j,u_j \rangle > (1-2\epsilon)\mu_L(U_t)
\end{equation}
for sufficiently large $n$.

From Theorem \ref{galkowski} and the bound $a\leq \mu_L(U)^{-1}$ that is immediate from semiclassical measures being probability measures, it follows that a full density subset of the remaining summands must be bounded above by $1+\epsilon$.

This implies that
\begin{equation}
(1+\epsilon)\frac{\#\{j\leq n: j\in B_t^c\}}{n} > (1-2\epsilon)\mu_L(U_t)
\end{equation}
for sufficiently large $n$.

Rearranging and passing to the limit $n\rightarrow \infty$ and then $\epsilon\rightarrow 0$, we obtain the required upper bound of
\begin{equation}
\limsup_{n\rightarrow\infty} \frac{\#\{j\leq n: j\in B_t^c\}}{n}\leq 1-\mu_L(U_t)=d(t).
\end{equation}

Hence $B_t$ is a density $d(t)$ sequence of eigenfunctions with semiclassical mass supported in the completely integrable region.
\end{proof}

\begin{prop}
\label{ergodicprop}
Let $A_t=\N\setminus B_t$. Then for each $t\in \mathcal{G}$, a full density subsequence of $(u_n)_{n\in A_t}$ equidistributes in $U_t$.
\end{prop}

\begin{proof}
From Theorem \ref{mainthm}, the sequence of eigenfunctions $(u_n)_{n\in B_t}$ has all semiclassical mass in the completely integrable region and $B_t$ has natural density $d(t)$.

Applying the local Weyl law again with the function $\chi_\epsilon$ from the proof of Theorem \ref{mainthm}, we obtain
\begin{equation}
\label{jt}
\frac{1}{n}\sum_{j\in [1,n]\cap B_t^c}\langle \chi_\epsilon(x,E_j^{-1/2}D) u_j,u_j \rangle > \mu_L(U_t)(1-\epsilon)
\end{equation}
for sufficiently large $n$.

Then, splitting the summation into the set 
\begin{equation}
A_{\epsilon,t}=\{j\in B_t^c:\langle \chi_\epsilon(x,E_j^{-1/2}D) u_j,u_j \rangle < 1-\sqrt{\epsilon}\}
\end{equation} 
and its complement, the upper bound of $1+\epsilon$ for a full density subset of the summands in \eqref{jt} then implies:
\begin{equation}
d_n(A_{\epsilon,t})<\mu_L(U_t)\cdot \frac{2\epsilon}{\epsilon+\sqrt{\epsilon}}
\end{equation}
using the notation $d_n$ from Lemma \ref{densitylemma}.
Passing to the limit $n\rightarrow \infty$, we obtain a subset of density exceeding $1-O(\sqrt{\epsilon})$ of $A_t$ with at least $\mu(U_t)>1-O(\sqrt{\epsilon})$ for any corresponding semiclassical measure.

An application of Lemma \ref{densitylemma} then gives us a full density subsequence of $A_t$ with all semiclassical mass in $U_t$.

Together with Theorem \ref{galkowski}, this implies that we can find a full density subsequence $(u_{n_k})$ of $A_t$ such that every associated semiclassical measure is of the form $1_{U_t}\cdot \mu_L(U_t)^{-1}\mu_L$.
\end{proof}

We now show that the set $(0,2]\setminus \mathcal{G}$ has Lebesgue measure $0$.

\begin{prop}
\label{propo}
If $(0,2]\setminus \mathcal{G}$ has positive Lebesgue measure, then there exists some $\epsilon>0$ and some interval $\mathcal{I}=[t_1,t_2]\subset (0,2]$ such that
\begin{equation}
\label{play1}
\frac{1}{|\mathcal{I}|}\int_\mathcal{I}\limsup_{n\rightarrow\infty}\left(\frac{\#\{j\in\N:E_j(t)\in \cup_{i=1}^n [\alpha_i^2-c,\alpha_i^2+c]\}}{n}\right) \, dt > 1+\epsilon
\end{equation}
for all $c>0$. Moreover, we can find such $\mathcal{I}$ with arbitrarily small length.
\end{prop}
\begin{proof}
By the monotone convergence property of measures, if $m((0,2]\setminus \mathcal{G})>0$ then there must exist $\epsilon>0$ and a positive measure set $S\subseteq (0,2]$ on which we have
\begin{equation}
\limsup_{n\rightarrow\infty}\left(\frac{\#\{j\in\N:E_j(t)\in \cup_{i=1}^n [\alpha_i^2-c,\alpha_i^2+c]\}}{n}\right)>1+2\epsilon
\end{equation}
for all $t\in S$ and for all $0<c<2/r_2^2$.
From the regularity of the Lebesgue measure, we can find an open set $S\subseteq U\subseteq (0,2]$ with $m(U)<m(S)+\delta$ for an arbitrarily small $\delta$.
We then have
\begin{equation}
\frac{1}{|S|}\int_S\limsup_{n\rightarrow\infty}\left(\frac{\#\{j\in\N:E_j(t)\in \cup_{i=1}^n [\alpha_i^2-c,\alpha_i^2+c]\}}{n}\right) \, dt > 1+2\epsilon
\end{equation}
and
\begin{equation}
\frac{1}{|U\setminus S|}\int_{U\setminus S}\limsup_{n\rightarrow\infty}\left(\frac{\#\{j\in\N:E_j(t)\in \cup_{i=1}^n [\alpha_i^2-c,\alpha_i^2+c]\}}{n}\right) \, dt \geq 1.
\end{equation}
from our pointwise bounds on the integrands.
By choosing $\delta$ sufficiently small, we are thus guaranteed the estimate
\begin{equation}
\frac{1}{|U|}\int_U \limsup_{n\rightarrow\infty}\left(\frac{\#\{j\in\N:E_j(t)\in \cup_{i=1}^n [\alpha_i^2-c,\alpha_i^2+c]\}}{n}\right) \, dt >1+\epsilon.
\end{equation}
Writing the open set $U$ as a countable union of disjoint open intervals, the average of the integrand over one such interval must exceed $1+\epsilon$, as claimed.

To complete the proof, we observe if we partition $\mathcal{I}$ into arbitrarily many intervals of equal length, at least one of them must also satisfy \eqref{play1}.
\end{proof}

To culminate the argument, we seek out a contradiction coming from the upper bound  \eqref{flowspeedeq} on speed of eigenvalue variation and the lower bound \eqref{play1} on the average proportion of eigenvalues lying in $c$-clusters.

\begin{prop}
For any $\epsilon>0$, there exists $c>0$ such that
\begin{equation}
\limsup_{m\rightarrow \infty} \frac{1}{m}\sum_{j=1}^m\frac{|\{t\in \mathcal{I}:E_j\in \cup_i [\alpha_i^2-c,\alpha_i^2+c]\}|}{|\mathcal{I}|}<d(t_1)+\epsilon
\end{equation}
for any sufficiently small interval $\mathcal{I}$.
\end{prop}
\begin{proof}
Note that we have the flow speed bound $\eqref{flowspeedeq}$ for a full density subsequence of eigenvalues, so if we can establish the claimed inequality for each summand with a sufficiently large index that obeys the flow speed bound, density will allow us to draw the desired conclusion.

We now suppose $E_j$ is a large eigenvalue that lies in this full density subsequence.

Writing $X=(A(t_1)^{-1}-A(t_2)^{-1})$ for brevity, Weyl's law applied to the mushroom gives
\begin{equation}E_j(t_1)-E_j(t_2)> (4\pi X -2\delta)j \end{equation}
for $\delta>0$ and all sufficiently large $j$.

Weyl's law for the semidisk (recalling that we constructed the completely integrable region quasimodes from semidisk eigenfunctions) gives us an upper bound of
\begin{equation}
\left(\frac{\pi r_2^2X}{2}+\delta\right) j
\end{equation}
for the number of quasi-eigenvalues in $[E_j(t_2),E_j(t_1)]$
and hence an upper bound of
\begin{equation}
2c\left(\frac{\pi r_2^2X}{2}+\delta\right) j
\end{equation}
for the length of $[E_j(t_2),E_j(t_1)]$ that lies within $\cup_i [\alpha_i^2-c,\alpha_i^2+c]$.

Now suppose that $E_j(t)$ spends proportion $q_j$ of $t\in\mathcal{I}$ in $\cup_i [\alpha_i^2-c,\alpha_i^2+c]$. 

From Proposition \ref{almostuniform}, it follows that the $q_j$ are uniformly bounded above by some $1-\delta$.

This means that we apply Corollary \ref{speedcor} to find a lower bound for the time taken by an eigenvalue $E_j$ in our full-density subsequence to traverse the set \\ \noindent${[E_j(t_2),E_j(t_1)]\setminus \cup_i [\alpha_i^2-c,\alpha_i^2+c]}$. Heuristically, we can think of this as dividing the size of this set by an upper bound for the speed of the eigenvalue's variation.


Precisely, we have
\begin{eqnarray}
(1-q_j)(t_2-t_1)&>& \frac{jX(4\pi-\pi r_2^2 c)-j\delta(1+2c) }{E_j(t_1)(\frac{\dot{A}(t_1)}{A(t_1)(1-d(t_1))}+\delta)}\\
&=&\frac{j}{E_j(t_1)}\cdot\left(\frac{X(4\pi-\pi r_2^2 c)-\delta(1+2c) }{\frac{\dot{A}(t_1)}{A(t_1)(1-d(t_1))}+\delta}\right)\\
&>& \left(\frac{4\pi}{A(t_1)}+\delta\right)^{-1}\cdot\left(\frac{X(4\pi-\pi r_2^2 c)-\delta(1+2c) }{\frac{\dot{A}(t_1)}{A(t_1)(1-d(t_1))}+\delta}\right)\\
&>& \frac{XA(t_1)^2(1-d(t_1))}{\dot{A}(t_1)}-\frac{\epsilon}{2}
\end{eqnarray}
where the final two lines follow from Weyl's law and passing to sufficiently small $\delta$ and $c$ respectively.

Additionally, since $A(t)$ is a linear polynomial in $t$, we have 
\begin{equation}
X=\frac{A(t_2)-A(t_1)}{A(t_1)A(t_2)}=\frac{(t_2-t_1)\dot{A}(t_1)}{A(t_1)A(t_2)}
\end{equation}
which implies that
\begin{eqnarray*}
1-q_j &>& \frac{A(t_1)}{A(t_2)}(1-d(t_1))-\frac{\epsilon}{2}\\
q_j &<& d(t_1)+(d(t_1)-1)\left(\frac{A(t_1)}{A(t_2)}-1\right)+\frac{\epsilon}{2}\\
&<& d(t_1)+\epsilon
\end{eqnarray*}
for sufficiently small $|\mathcal{I}|$, using the uniform continuity of $A$.

Thus we have the required inequality for all sufficiently small intervals $\mathcal{I}$ and all sufficiently large $j$ in a full density subsequence on which \eqref{flowspeedeq} holds.

\end{proof}

\begin{prop}
\label{propo2}
For any $\epsilon>0$ there exists $c>0$, such that
\begin{equation}
\frac{1}{|\mathcal{I}|}\int_\mathcal{I}\limsup_{n\rightarrow\infty}\left(\frac{\#\{j\in\N:E_j(t)\in \cup_{i=1}^n [\alpha_i^2-c,\alpha_i^2+c]\}}{n}\right)\, dt<1+\epsilon
\end{equation}
for all sufficiently small $|\mathcal{I}|$.
\end{prop}
\begin{proof}
By the dominated convergence theorem, it suffices to show that we can find $c$ such that
\begin{equation}
\frac{1}{|\mathcal{I}|}\int_\mathcal{I} \frac{\#\{j:E_j(t)\in \cup_{i=1}^n[\alpha_i^2-c,\alpha_i^2+c]\}}{n}\, dt < 1+\frac{\epsilon}{2}
\end{equation}
for sufficiently large $n$.
This quantity is bounded above by
\begin{equation}
\frac{1}{n}\sum_{j:E_j(t_2)<\alpha_n^2+c} \frac{|\{t\in \mathcal{I}:E_j\in \cup_i [\alpha_i^2-c,\alpha_i^2+c]\}|}{|\mathcal{I}|}=\frac{1}{n}\sum_{j:E_j(t_2)<\alpha_n^2+c} q_j.
\end{equation}
The sum is controlled by the previous proposition, giving us an upper bound of
\begin{equation}
\frac{1}{n}\cdot \max \{j:E_j(t_2)<\alpha_n^2+c\}(d(t_1)+\delta)
\end{equation}
for sufficiently large $n$.

From Weyl's law, we have
\begin{equation}
\max \{j:E_j(t_2)<\alpha_n^2+c\}<(\alpha_n^2+c)\left( \frac{A(t_2)}{4\pi}+\delta\right)
\end{equation}
for sufficiently large n.

By taking $\delta$, $c$, and $\mathcal{I}$ sufficiently small, we then obtain 
\begin{equation}
\frac{1}{n}\sum_{j:E_j(t_2)<\alpha_n^2+c} q_j < \left(\frac{\alpha_n^2}{n}\cdot\frac{A(t_2)d(t_2)}{4\pi}\right)+\frac{\epsilon}{4}.
\end{equation}

Inverting the estimate \eqref{quasiasymptoticeqn} provides an upper bound of $1+\frac{\epsilon}{4}$ for the first summand on the right-hand side for all sufficiently large $n$, thus completing the proof.
\end{proof}

\begin{cor}
The set $\mathcal{G}$ has full measure in $(0,2]$.
\end{cor}
\begin{proof}
This is an immediate consequence of Propositions 5.5 and 5.7.
\end{proof}

\appendix
\section{}
In this appendix, we prove the following abstract lemma that we have used several times to assemble full density subsequences along which a given function has limit $0$.

\begin{lem}
\label{densitylemma}
If there exists a function $g:\mathbb{N}\rightarrow \R$ and a family of subsets $S_j\subset \mathbb{N}$ such that
\begin{equation}
\liminf_{n\rightarrow\infty} \frac{\#\{k\leq n:k\in S_j\}}{n}>d-\epsilon_j
\end{equation}
and 
\begin{equation}
\limsup_{n\in S_j\rightarrow\infty} g(n)<\epsilon_j'
\end{equation}
where $\epsilon_j,\epsilon_j'\searrow 0$, then there exists a subset $S\subset \mathbb{N}$ such that
\begin{equation}
\liminf_{n\rightarrow\infty} \frac{\#\{k\leq n:k\in A\}}{n}\geq d
\end{equation}
and
\begin{equation}
\label{gconv}
\lim_{n\in S \rightarrow\infty} g(n)=0.
\end{equation}
\end{lem}
\begin{proof}
For ease of notation, we define
\begin{equation}
d_n(A)=\frac{\#\{k\leq n:k\in A\}}{n}
\end{equation}
for $A\subseteq\mathbb{N}$ and $n\in\mathbb{N}$.

We have $g(n)< 2\epsilon_j'$ for cofinitely many elements of $S_j$, and we denote these sets by $S_j'$.
Now let
\begin{equation}
B_j=\{k\in\mathbb{N}:g(k)\geq 2\epsilon_j'\}\subseteq \mathbb{N}\setminus S_j'.
\end{equation}
Since each $d_n$ respects the partial ordering of set inclusion and is additive with respect to disjoint unions, we can construct a strictly increasing sequence $(N_j)_{j\in\mathbb{N}}$ such that $N_1=1$ and $d_n(B_j)< 1-d+2\epsilon_j$ for all $n\geq N_j$.

We define
\begin{equation}
B=\bigcup_{j\in \mathbb{N}}B_j\cap [N_j,\infty).
\end{equation}

If $n\in [N_j,N_{j+1})$, then any $k\in [1,n]\cap B$ must lie in $B_i$ for some $i\leq j$ and hence in $B_j$.

This implies that for $n\in [N_j,N_{j+1})$ we have $d_n(B)\leq d_n(B_j) <1-d+2\epsilon_j$ and consequently, that $\displaystyle \limsup_{n\rightarrow\infty} d_n(B)\leq 1-d$.

We now take $S:=\mathbb{N}\setminus B$, with the required density bound 

\begin{equation}
\liminf_{n\rightarrow\infty }d_n(S)\geq d.
\end{equation}

To complete the proof we observe that if $n\in [N_j,\infty) \cap S$, then $n\in \mathbb{N}\setminus B_i$ for each $i\leq j$, and hence $g(n)<2\epsilon_j'$. This establishes \eqref{gconv}.
\end{proof}

\bibliographystyle{plain}
\bibliography{mybib}

\end{document}